\documentclass[11pt]{amsart}

\usepackage[margin=3.5cm]{geometry}
\usepackage{amssymb, amsmath}
\usepackage{amscd}
\usepackage{hyperref}
\usepackage[alwaysadjust]{paralist}
\usepackage[capitalize]{cleveref}
\usepackage{enumitem}
\usepackage{esint}
\usepackage{graphicx} 
\usepackage[rgb]{xcolor} 
\usepackage{verbatim}

\numberwithin{equation}{section}

\newtheorem{lemma}[equation]{Lemma}
\newtheorem{prop}[equation]{Proposition}
\newtheorem{theorem}[equation]{Theorem}
\theoremstyle{definition}

\def\IR{\mathbb R}

\def\IZ{\mathbb Z}

\def\eps{\varepsilon}

\newcommand{\inj}{\operatorname{inj}}
\newcommand{\im}{\operatorname{im}}
\newcommand{\area}{\operatorname{area}}
\newcommand{\C}{\mathcal{C}}
\newcommand{\Sph}{\mathbb{S}}

\begin{document}

\title[Maximizing the first eigenvalue on non-orientable surfaces]{Existence of metrics maximizing the first eigenvalue on non-orientable surfaces}

\author{Henrik Matthiesen}
\address
{HM: Max Planck Institute for Mathematics,
Vivatsgasse 7, 53111 Bonn
\newline
current adress: Department of Mathematics, University of Chicago,
5734 S. University Ave, Chicago, Illinois 60637}
\email{hmatthiesen@math.uchicago.edu}
\author{Anna Siffert}
\address
{AS: Max Planck Institute for Mathematics,
Vivatsgasse 7, 53111 Bonn}
\email{siffert@mpim-bonn.mpg.de}
\date{\today}

\begin{abstract} We prove the existence of metrics maximizing the first eigenvalue normalized by area on closed, non-orientable surfaces assuming two spectral gap conditions.
These spectral gap conditions are proved by the authors in \cite{MS3}.
\end{abstract}

\maketitle

\section{Introduction}

For a closed Riemannian surface $(\Sigma,g)$ the spectrum of the Laplace operator
acting on smooth functions, is purely discrete and can be written as
\begin{equation*}
0=\lambda_0<\lambda_1(\Sigma,g) \leq \lambda_2(\Sigma,g) \leq \lambda_3(\Sigma,g) \leq \dots \to \infty,
\end{equation*}
where we repeat an eigenvalue as often as its multiplicity requires.

\smallskip

The pioneering work of Hersch \cite{hersch} and Yang--Yau \cite{yang_yau} raised the natural question, whether there are metrics $g$ that maximize the scale-invariant quantities
\begin{equation*}
\bar \lambda_1(\Sigma):=\lambda_1(\Sigma,g)\area(\Sigma,g)
\end{equation*}
if $\Sigma$ is a closed surface of fixed topological type (see also \cite{karpukhin,li_yau} for the case of non-orientable surfaces). 
Such maximizers have remarkable properties.
In fact, they always arise as immersed minimal surfaces (of possibly high codimension) in a sphere \cite{Ilias_ElSoufi} and are unique in their conformal class unless they are branched immersions into the two sphere \cite{ckm, montiel_ros, nayatani_shoda}.
By a slight abuse of notation, we also call
$\Sigma$, endowed with a maximizing metric, a \lq maximizer\rq.  

\smallskip

For the statement of our results and related work, we need to introduce some notation.
We write $\Sigma_\gamma$ for a closed orientable surface of genus $\gamma.$
Similarly, $\Sigma^K_\delta$ denotes a closed non-orientable surface of non-orientable genus $\delta.$ Here, $K$ stands for Klein.
We briefly elaborate on these notions in \cref{notation}. Furthermore, we use the common notation
\begin{align*}
\Lambda_1(\gamma)= \sup_{g} \lambda_1(\Sigma_\gamma,g)\area(\Sigma_\gamma,g),
\end{align*}
and similarly,
\begin{align*}
\Lambda_1^K(\delta)=\sup_{g} \lambda_1(\Sigma_\delta^K,g)\area (\Sigma_\delta^K,g),
\end{align*}
with the supremum taken over all smooth metrics on $\Sigma_\gamma,$ respectively $\Sigma_\delta^K.$

\smallskip

Explicit values for $\Lambda_1(\gamma)$ or $\Lambda_1^K(\delta)$ are only known in very few cases.
However, in all of these cases not only the values but also explicit maximizing metrics are known.

The case of the sphere is due to Hersch.
We have $\Lambda_1(\Sph^2)=8 \pi$ with unique maximizer the round metric \cite{hersch}.
His arguments are very elegant and a cornerstone in the development of the subject. For the real projective plane, we have
$\Lambda_1(\mathbb{RP}^2)=12 \pi$ with unique maximizer the round metric \cite{li_yau}.
The proof extends the ideas from \cite{hersch} in a conceptually very nice way.

The first result for higher genus surfaces is due to Nadirashvili, namely
$\Lambda_1(T^2)=8 \pi^2/\sqrt{3}$ with unique maximizer the flat equilateral torus \cite{nadirashvili}.
Nadirashvili's arguments are very different from the previously employed methods.
The crucial step in his proof is to obtain the existence of a maximizer.
Using \cite{montiel_ros} (see also \cite{ckm}) it follows that such a maximizer necessarily has to be flat.
The sharp bound follows then from earlier work of Berger \cite{berger}.

For the Klein bottle,  
$\Lambda_1(K)=12 \pi E (2 \sqrt{2}/3)$ with unique maximizer a metric of revolution \cite{ckm,esgj, nadirashvili}.
Here $E$ is the complete elliptic integral of the second kind.

There is also a conjecture concerning the sharp bound on genus $2$ surfaces \cite{jlnnp}, a proof of which has very recently been given by Nayatani and Shoda 
in \cite{nayatani_shoda}.

\smallskip

Let us also mention that there are a quite some results concerning similar questions for higher order eigenvalues, 
-- see e.g.\
\cite{nadirashvili_sire_2, petrides-2} and
 \cite{nadirashvili-2, nadirashvili_sire_3, knpp}
for the case of  $\Sph^2$
and \cite{nadirashvili_penskoi, karpukhin-3} for the case of $\IR P^2$.
\smallskip

The growing interest in finding
maximizers for eigenvalue functionals on surfaces starting from Nadirashvili's paper \cite{nadirashvili} 
is certainly connected to the connection of the problem to minimal surfaces in spheres.
Similarly, for the Steklov eigenvalue problem, there is a connection to free boundary minimal
surfaces in Euclidean balls.
In pioneering work 
Fraser and Schoen showed the existence of maximizers for the first Steklov eigenvalue on surfaces with boundary of genus $0$ \cite{fs}.

\smallskip

Recently, Petrides used many of the ideas in \cite{fs} to prove the following beautiful result concerning metrics realizing $\Lambda_1(\gamma).$

\begin{theorem}[Theorem 2 in \cite{petrides}] \label{max_eigen_petrides}
If $\Lambda_1(\gamma-1)<\Lambda_1(\gamma)$, there is a metric $g$ on $\Sigma=\Sigma_\gamma$, which is smooth
away from finitely many conical singularities, such that 
\begin{equation*}
\lambda_1(\Sigma,g)\area(\Sigma,g)=\Lambda_1(\gamma).
\end{equation*}
\end{theorem}

We extend this to non-orientable surfaces.
Since non-orientable surfaces can degenerate to non-orientable surfaces as well as orientable ones, we need to make two instead of only
a single spectral assumption.

\begin{theorem} \label{ex_non_orientable_main}
If $\Lambda_1^K(\delta-1)<\Lambda_1^K(\delta)$ and $\Lambda_1(\lfloor (\delta -1)/2 \rfloor) < \Lambda_1^K(\delta)$,
there is a metric $g$ on $\Sigma=\Sigma_\delta^K$, which is smooth away from at most finitely many conical singularities, such that
\begin{equation*}
\lambda_1(\Sigma,g)\area(\Sigma,g)=\Lambda_1^K(\delta).
\end{equation*}
\end{theorem}

Our methods are very similar to those in \cite{petrides}.
In addition to the cases already handled by Petrides, we also need to take care of degenerating one-sided geodesics.

\medskip

The non-strict inequality $\Lambda_1(\gamma-1) \leq \Lambda_1(\gamma)$ was proved by Colbois and El Soufi in \cite{ces} using a result of Ann{\'e} \cite{anne}.
It is easy to obtain the non-strict versions of the inequalities assumed in \cref{ex_non_orientable_main} along the same lines, 
see also \cite{MS2}.
In \cite{MS2} we prove the monotonicity $\Lambda_1(\gamma-1) < \Lambda_1(\gamma)$ under some extra assumptions on the maximizing metric using a relatively simple glueing construction.
In \cite{MS3},
by means of a much more complicated glueing construction and using \cref{max_eigen_petrides} and \cref{ex_non_orientable_main}, we prove all of the spectral gap conditions assumed in \cref{max_eigen_petrides} and \cref{ex_non_orientable_main} above 
In particular this implies the existence of maximizing metrics on closed surfaces of any topological type.

\begin{theorem}[Theorem 1.3 in \cite{MS3} using \cref{max_eigen_petrides} and \cref{ex_non_orientable_main}]
Let $\Sigma$ be a closed surface. Then there is a metric $g$ on $\Sigma$, which is smooth away from at most
finitely many conical singularities, such that
$$
\lambda_1(\Sigma,h) \area(\Sigma,h) \leq \lambda_1(\Sigma,g) \area(\Sigma,g)
$$
for any smooth metric $h$ on $\Sigma$.
\end{theorem}

\smallskip

\textbf{Outline.}
In Section\,\ref{comp} we prove that the set of orientable, hyperbolic surfaces with injectivity radius bounded below is a compact
subset of the moduli space. This is a version of the Mumford compactness criterion for non-orientable surfaces.
We use this result in the main section, namely Section\,\ref{maximize}, in which we prove
 Theorem\,\ref{ex_non_orientable_main}.
\smallskip

\textbf{Acknowledgements.}
The authors would like to thank the Max Planck Institute for Mathematics in Bonn for financial support and excellent working conditions.

\section{Compactness for non-orientable surfaces} \label{comp}
The Mumford compactness criterion \cite{mumford}  states that the set of orientable, hyperbolic surfaces with injectivity radius bounded below is a compact
subset of the moduli space. 
In this section we show that this also holds for non-orientable surfaces.
Probably, this is well-known, but for the sake of completeness and since we will use the arguments from our proof again, we include a proof below.

\smallskip

Given any Riemannian metric $g_0$ on $\Sigma=\Sigma_\delta^K,$  the Poincar\'e uniformization theorem asserts that we can find a new metric on $\Sigma$
which is conformal to $g_0$ and has constant curvature $+1,0,$ or $-1,$ depending
on the sign of $\chi(\Sigma).$
Assuming $\delta \geq 3,$ these metrics have curvature $-1.$
Let $h_k$ be a sequence of such metrics on $\Sigma$ with injectivity radius bounded
uniformly from below, $\inj(\Sigma,h_k) \geq c>0.$
The goal is to prove that there exist diffeomorphisms $\sigma_k$ of $\Sigma$ and a hyperbolic metric $h$ of $\Sigma$, such that $\sigma_k^* h_k$ converges smoothly to $h$ as $k\rightarrow\infty$.
Our strategy is to apply the Mumford compactness criterion to the orientation double covers
of the surfaces $(\Sigma,h_k).$

\smallskip

Consider the orientation double cover $\hat \Sigma=\Sigma_{\delta-1}$ of $\Sigma$  
endowed with the pullback metrics of $h_k,$ denoted by $\hat h_k.$
Since $\delta \geq 3,$ these are orientable hyperbolic surface of genus $\delta -1$ and may
thus be regarded (if we also fix a marking) as elements in Teichm{\"u}ller space $\mathcal{T}_{\delta -1},$
which in addition admit fixed point free, isometric, orientation reversing involutions $\iota_k.$

\smallskip

We have the following lemma.

\begin{lemma} \label{limit_inv}
Assume that $\inf_k \inj(\Sigma_\delta^K, h_k)>0.$
Then there exist a sequence of diffeomorphisms $\tau_k \colon \Sigma_{\delta -1} \to \Sigma_{\delta -1},$
such that, up to taking a subsequence, $\tau_k^* \hat h_k \to \hat h$ in $C^{\infty}.$
Moreover, $(\Sigma_{\delta -1},\hat h)$ admits a fixed point free, isometric, orientation reversing involution $\iota$, which is obtained as a $C^1$-limit of the involutions
$\tau_k^{-1} \circ \iota_k \circ \tau_k.$
\end{lemma}

\begin{proof}
As above, we simply write $\Sigma$ instead of $\Sigma_{\delta}^K,$ and $\hat \Sigma$ instead of $\Sigma_{\delta-1}.$
It is elementary to see that $\inj(\hat \Sigma,\hat h_k) \geq \inj(\Sigma ,h_k).$
Therefore, we can apply the Mumford compactness criterion \cite{mumford} and find
diffeomorphisms $\tau_k$ and a limit metric $\hat h$ as asserted.

\smallskip

It remains to show that we can find the involution $\iota.$
Since $\tau_k^* \hat h_k \to \hat h$ in $C^{\infty},$ we have the uniform Lipschitz bound
\begin{equation*}
\begin{split}
& d_{\hat h} ((\tau_k^{-1} \circ  \iota_k \circ \tau_k)(p),(\tau_k^{-1} \circ \iota_k \circ \tau_k)(q)) 
\\
& \leq C d_{\tau_k^* \hat h_k} ((\tau_k^{-1} \circ \iota_k \circ \tau_k)(p),(\tau_k^{-1} \circ \iota_k \circ \tau_k)(q))
\\ 
&=C d_{\tau_k^* \hat h_k} (p,q)
\\
&\leq C d_{\hat h}(p,q).
\end{split}
\end{equation*}
Similarly, we obtain a uniform bound on $\| D \iota\|_{C^{0,1}(\hat \Sigma, \hat h)}$.
Since $\hat \Sigma$ is compact, it follows from Arz\'ela--Ascoli, that, up to taking
a subsequence, $\tau_k^{-1} \circ \iota_k \circ \tau_k \to \iota$ in $C^{1}(\hat \Sigma , \hat h).$
We have
\begin{equation} \label{dist_pres}
\begin{split}
d_{ \hat h}(\iota(p),\iota(q)) 
 & \leq  \lim_{k \to \infty} d_{\tau_k^*\hat h_k}(\iota(p),(\tau_k^{-1} \circ \iota_k \circ \tau_k)(p))  
\\
& + \lim_{k \to \infty} d_{\tau_k^*\hat h_k}((\tau_k^{-1} \circ \iota_k \circ \tau_k)(p), (\tau_k^{-1} \circ \iota_k \circ \tau_k)(q))
\\ 
& + \lim_{k \to \infty} d_{\tau_k^*\hat h_k}((\tau_k^{-1} \circ \iota_k \circ \tau_k)(q),\iota(q))
\\
& \leq C \lim_{k \to \infty} d_{C^0(\hat \Sigma,\hat h)} (\tau_k^{-1} \circ \iota_k \circ \tau_k,\iota)
\\
& + \lim_{k \to \infty} d_{\tau_k^*\hat h_k}((\tau_k^{-1} \circ \iota_k \circ \tau_k)(p), (\tau_k^{-1} \circ \iota_k \circ \tau_k)(q))
\\
& = d_{\hat h}(p,q),
\end{split}
\end{equation}
using that $\tau_k^*\hat h_k \to \hat h$ in $C^\infty,$ and $\tau_k^{-1} \circ \iota_k \circ \tau_k \to \iota$ in $C^0(\hat \Sigma,\hat h).$
Observe that $\iota$ is an involution again, hence \eqref{dist_pres} implies that actually
\begin{equation*}
d_{\hat h}(\iota(p),\iota(q))  = d_{\hat h}(p,q).
\end{equation*}
By the Myers--Steenrod theorem it thus follows that $\iota$ is a smooth, isometric involution.

\smallskip

We need to show that $\iota$ does not have any fixed points.
But this is a consequence of the general bound $d_{\tau_k^*\hat h_k}((\tau_k^{-1} \circ \iota_k \circ \tau_k)(p),p)\geq c >0$ for some
uniform $c.$
To prove this let $c>0$ be such that $B_{\hat h}(x,2c)\subset \hat \Sigma$ is strictly geodesically convex
for any $x \in \hat \Sigma.$
Then $B_{\tau_k^* \hat h_k}(x,c)$ is strictly geodesically convex for $k \geq K$ sufficiently large.
Assume now that there is  $k \geq K,$  such that $d_{\hat h_k}((\tau_k^{-1} \circ \iota_k \circ \tau_k)(p),p) <c.$
Let $\gamma$ be the unique minimizing geodesic connecting $p$ to $(\tau_k^{-1} \circ \iota_k \circ \tau_k)(p).$
Since $\tau_k^{-1} \circ \iota_k \circ \tau_k$ is an isometry, we need to have
$\im (\tau_k^{-1} \circ \iota_k \circ \tau_k \circ \gamma) = \im \gamma.$
Since $\iota_k$ is fixed point free, $\gamma$ is non-constant.
Therefore, $\tau_k^{-1} \circ \iota_k \circ \tau_k$   restricted to $\im \gamma$ induces
an involution of the interval $[0,1]$ mapping $0$ to $1$ and vice versa.
But such an involution needs to have a fixed point.
It follows that $\iota_k$ has a fixed point for large $k$, which is a contradiction.

Finally, note that $\iota$ is orientation reversing by $C^{0}$-convergence.
\end{proof}

Since $\tau_k^{-1} \circ \iota_k \circ \tau_k \to \iota$ in $C^1$ it follows that the metric $\hat h$ on $\hat \Sigma$ is $\iota$-invariant.
Therefore, it induces a smooth hyperbolic metric $h$ on $\Sigma.$
Moreover, the hyperbolic metrics on $\Sigma$ induced from $\tau_k^* h_k$ and $\tau_k^{-1} \circ \iota_k \circ \tau_k$ converge 
smoothly to $h$ on $\Sigma.$
Finally, observe that the diffeomorphisms $\tau_k$ induce diffeomorphisms $\sigma_k$ of $\Sigma,$
such that $\sigma_k^* h_k$ are the metrics described above and converge smoothly to $h.$

Thus we have proved the following proposition.

\begin{prop} \label{comp_non_or}
Let $(h_k)$ be a sequence of hyperbolic metrics on $\Sigma_\delta^K$ such that their injectivity radius is uniformily bounded from below
$\inj(\Sigma_\delta^K,h_k) \geq c>0.$
Then there are diffeomorphisms $\sigma_k$ of $\Sigma_\delta^K$ and a hyperbolic metric $h,$
such that $\sigma_k^* h_k \to h$ smoothly.
\end{prop}

\section{Maximizing the first eigenvalue} \label{maximize}
In this section we extend \cite[Theorem 2]{petrides} to the non-orientable case.
The strategy is the same as in \cite{petrides}.
That is, we first use that we can maximize the first eigenvalue in each conformal class.
We then pick a maximizing sequence, consisting of maximizers in their own conformal class.
This has the advantage, that these metrics can be studied in terms of sphere-valued harmonic maps.
Using these harmonic maps it is possible to estimate the first eigenvalue along the maximizing sequence in case that the conformal class
degenerates.
To do so, we extend the results from \cite{zhu} to non-orientable surfaces.

\smallskip

For fixed non-orientable genus $\delta \geq 3,$ let $c_k$ be a sequence of conformal classes
on $\Sigma=\Sigma_\delta^K$
 represented by hyperbolic metrics $h_k,$ such that
\begin{equation*}
\lim_{k \to \infty} \sup_{g \in c_k} \lambda_1(\Sigma ,g) \area (\Sigma,g) = \Lambda_1^{K}(\delta).
\end{equation*}
We will now use the following result due to Nadirashvili--Sire (with an extra assumption not relevant for our purposes) and, independently, Petrides.

\begin{theorem}[{\cite[Theorem 2.1]{nadirashvili_sire} or \cite[Theorem 1]{petrides}}]
For each conformal class $c_k$ as above, there is a metric $g_k,$ which is smooth away from finitely many conical singularities, such that
\begin{equation*}
\lambda_1(\Sigma,g_k)\area(\Sigma,g_k)= \sup_{g \in c_k} \lambda_1(\Sigma,g)\area(\Sigma,g).
\end{equation*}
\end{theorem}

From now on we assume that $g_k \in c_k$ is picked as in the preceding theorem.
Moreover, we assume that they are normalized to have
\begin{equation*}
\area(\Sigma,g_k)=1.
\end{equation*}
Since these metrics are maximizers, there is a family of first eigenfunctions $u_1^k,\dots, u_{\ell(k)+1}^k,$
such that $\Phi_k=(u_1^k , \dots, u_{\ell(k)+1}^k) \colon (\Sigma,h_k) \to \Sph^{\ell(k)}$
is a harmonic map \cite{elsoufi}.
Since the multiplicity of $\lambda_1$ is uniformly bounded in terms of the topology of $\Sigma$ \cite{besson, cheng},
we may pass to a subsequence, such that $\ell(k)$ is some constant number $l$.
Moreover, in this situation the maximizing metrics can be recovered by
\begin{equation*}
g_k= \frac{|\nabla \Phi_k |^2_{h_k}}{\lambda_1(\Sigma,g_k)}h_k.
\end{equation*}
In view of Proposition\,\ref{comp_non_or}, we want to show the following proposition.

\begin{prop} \label{degen_prop}
The injectivity radius of
$h_k$ is uniformly bounded from below, provided  that
$\Lambda_1^{K}(\delta)>\Lambda_1(\delta-1),$
and $\Lambda_1^{K}(\delta)>\Lambda_1^{K}(\delta-1).$
\end{prop}

We will argue by contradiction and assume $\inj{(\Sigma, h_k)} \to 0.$
The Margulis lemma implies that we can find closed geodesics $\gamma_k^1,\dots,\gamma_k^s$ in $(\Sigma ,h_k),$
such that their lengths go to zero, i.e. $l_{h_k}(\gamma_i^k) \to 0,$ as $k \to \infty$.
We assume that $s$ is chosen maximal with this property.

Each of these geodesics is either one-sided or two-sided.
If a such a geodesic is two-sided, tubular neighborhoods are described by the classical
collar lemma for hyperbolic surfaces \cite{buser}.
In the second case we may apply the collar lemma to the orientation double cover as follows.

Let $c$ be a one-sided closed geodesic in $\Sigma.$
We write $\hat \Sigma$ for the orientation double cover and $\tau$ for the non-trivial deck transformation.
The lifts of $c$ to $\hat \Sigma$ can not be closed, since in this case they would be disjoint and it would follow
that $c$ is two-sided.
Thus the lifts $c_1$ and $c_2$ are geodesic segments with $\tau \circ c_1=c_2.$
Let $\mathcal{C}$ be a collar around the closed geodesic $c_2 * c_1.$
It is not very difficult to see that the action of $\tau$ near $c_2*c_1$ is just given by rotation about $\pi$ and reflection at $c_2 * c_1.$
Therefore, $\tau$ maps $\mathcal{C}$ to itself (by the explicit construction of $\mathcal{C}$), so that we can use $\mathcal{C}/\tau$ as a tubular neighborhood of $c$.

\smallskip

Our first goal is to prove that for the situation at hand the volume, measured with respect to $g_k$, either concentrates in the neighborhood of a pinching geodesic, or in one connected component of the complement of these neighborhoods.
Before stating and proving this result we need to introduce some notation, which we borrow from 
 Section\,4 in \cite{petrides}.

\smallskip

We write $s_1$ for the number of one-sided closed geodesics with length going to $0$.
Moreover, we denote by $s_2$ the number of such geodesics that are two-sided.
Clearly, $s=s_1+s_2$ and $0\leq s_1,s_2\leq s$.
From now on we assume that the closed geodesics $\gamma_{k}^{i}$ are ordered such that the first $s_1$ geodesics are one-sided.
Moreover, we write $l_k^i=l_{h_k}(\gamma_k^i)$ for the hyperbolic length of the short geodesics.

\smallskip

For all $s_1+1\le i\le s$ the collar theorem \cite{buser} asserts the existence of an open neighborhood $P_k^i$ of $\gamma_k^i$ isometric to the following truncated hyperbolic cylinder
\begin{equation*}
{\mathcal C}_k^i = \left\{\left(t,\theta\right)\lvert\ -w_k^i<t<w_k^i,\, 0\le \theta<2\pi\right\}
\end{equation*}
with 
\begin{equation*}
w_k^i = \frac{\pi}{l_k^i}\left(\pi-2 \arctan\left(\sinh\frac{l_k^i}{2}\right)\right)
\end{equation*}
endowed with the metric
\begin{equation*}
\left(\frac{{l_k^i}}{2\pi \cos\left(\frac{l_k^i}{2\pi}t\right)}\right)^2(dt^2 +d\theta^2).
\end{equation*}
Below we identify $(\theta,t)=(0,t)$ with $(\theta,t)=(2\pi,t)$. Thus the closed geodesic $\gamma_k^i$ corresponds to $\left\{t=0\right\}$. 

\smallskip

By the discussion above and the the collar theorem again, we get that for all $1 \le i\le s_1$, there exists an open neighborhood $P_k^i$ of $\gamma_k^i$ isometric to the following truncated M\"obius strip
\begin{equation*}
\mathcal{M}_k^i = \left\{\left(t,\theta\right)\lvert\ -w_k^i<t<w_k^i,\, 0\le \theta<2\pi\right\}/\sim
\end{equation*}
with 
\begin{equation*}
w_k^i = \frac{\pi}{2l_k^i}\left(\pi-2 \arctan\left(\sinh l_k^i\right)\right)
\end{equation*}
endowed with the metric
\begin{equation*}
 \left(\frac{2 l_k^i}{2\pi \cos\left(\frac{2 l_k^i}{2 \pi}t\right)}\right)^2 \left(dt^2 +d\theta^2\right).
\end{equation*}
Moreover, the equivalence relation $\sim$ is given by identifying $(t, \theta,) \sim (-t, \theta+\pi)$, where $\theta + \pi \in \IR / 2 \pi \IR$. 
Hence, the closed geodesic $\gamma_k^i$ corresponds to $\left\{t=0\right\}$. 

\smallskip 

We denote by $\Sigma_{k}^1,\cdots,\Sigma_{k}^r$ the connected components of $\Sigma \setminus \bigcup_{i=1}^sP_{k}^i$.
Consequently, $\Sigma$ can be written as the disjoint union
\begin{equation*}
\Sigma=\left(\bigcup_{i=1}^s P_k^i \right)\bigcup \left(\bigcup_{j=1}^r \Sigma_k^j\right).
\end{equation*}
For $s_1+1\le i\le s$ and $0<b<w_k^i$ we denote by $P_k^i\left(b\right)$ the truncated hyperbolic cylinder whose length, compared to $P_k^i$,
is reduced by $b$, i.e.,
\begin{equation*}
P_k^i\left(b\right)=\left\{\left(t,\theta\right),\, -w_k^i+b < t < w_k^i -b\right\}.
\end{equation*}
Analogously, for $1\le i\le s_1$ and $0<b<w_k^i$, we introduce
\begin{equation*}
P_k^i\left(b\right)=\left\{\left(t,\theta\right),\, -w_k^i+b < t < w_k^i -b\right\}/\sim.
\end{equation*}
Finally, we denote by $\Sigma_k^j\left(b\right)$ the connected components of ${\displaystyle \Sigma\setminus \bigcup_{i=1}^s P_k^i(b)}$ which contains $\Sigma_k^j$. 

\smallskip

We are now ready to prove the above mentioned result, namely, that the volume either concentrates in the neighborhood of a pinching geodesic $P_k^i$, or in one connected component $\Sigma_k^j$ of the complement of these neighborhoods.

\begin{lemma} \label{vol_cases}
There exists $D>0$ such that one of the two following assertions is true:

\begin{enumerate}
\item There exists an $i\in \left\{1,\dots,s\right\}$ such that 
$$\area_{g_k}\left(P_k^i\left(a_k\right)\right)\ge 1 -\frac{D}{a_k}$$
for all sequences $a_k\to +\infty$ with $\frac{a_k}{w_k^i}\to 0$ as $k\to +\infty$ for all $1\le i\le s$.
\item There exists a $j\in \left\{1,\dots,r\right\}$ such that 
$$\area_{g_k}\left(\Sigma_k^j\left(9 a_k\right)\right)\ge 1 -\frac{D}{a_k}$$
for all sequences $a_k\to +\infty$ with $\frac{a_k}{w_k^i}\to 0$ as $k\to +\infty$ for all $1\le i\le s$.
\end{enumerate}
\end{lemma}

\begin{proof}
The proof of Claim\,11 in \cite{petrides} can easily be adapted to the present situation.
First recall the rough strategy of the proof: construct suitable test functions for $\lambda_1(\Sigma,g_k)$ in the $P_k^i$ and the $\Sigma_k^j$'s, apply the min-max formula for the first eigenvalue and prove the claim by contradiction.
More precisely, on $\hat \Sigma,$ the test functions are constructed with linear decay in the $t$ variable in neck regions of the type
$\hat P_k^i(2a_k)\setminus \hat P_k^i(3a_k)$ and $\hat P_k^i(1 a_k) \setminus \hat P_k^i(2a_k),$ respectively,
where the hat indicates that we consider the preimages under the covering map $\hat \Sigma \to \Sigma.$
By conformal invariance, the Dirichlet energy of these can be estimated using the hyperbolic metric and decays like $a_k^{-1}.$
From the construction it is clear that these functions are invariant under the relevant involutions.
From this point on, one can just follow the arguments in \cite{petrides}.
\end{proof}

Below we consider the two possible cases of the preceding lemma separately.
The following lemma deals with the first case, i.e.\ when the volume concentrates in one of the $P_{k}^i$.
We show that in this case we would have $\Lambda_1^K(\delta)\leq 8\pi$ if $\gamma_{k}^i$ is $2$-sided; and
$\Lambda_1^K(\delta)\leq 12\pi$ if $\gamma_{k}^i$ is $1$-sided.

\begin{lemma} \label{geod_conc}
Suppose that there exists an $i\in\{1,...,s\}$ such that 
\begin{equation*}
\area_{g_{k}}(P_{k}^i(a_{k}))\geq 1-\frac{D}{a_{k}}
\end{equation*}
for all sequences $a_k\to +\infty$ with $\frac{a_k}{w_k^i}\to 0$ as $k\to +\infty$ for all $1\leq i\leq s$.
\begin{enumerate}
\item If $\gamma_{k}^i$ is $2$-sided, then  $\Lambda_1^K(\delta)\leq 8\pi$.
\item  If $\gamma_{k}^i$ is $1$-sided, then $\Lambda_1^K(\delta)\leq 12\pi$.
\end{enumerate}
\end{lemma}

\begin{proof}
In \cite{petrides}, Petrides proved the first statement by following ideas of Girouard \cite{girouard}.
The proof of the second statement is carried out analogously.

\smallskip

By assumption, there exists an $i\in\{1,...,s\}$, such that the volume concentrates on $P_{k}:=P_{k}^i$.
On $P_k$ we have coordinates $(t,\theta)$ as above (on $\mathcal{M}_k$).
By the assumptions on the volume and $a_k$, we can find cut-off functions $\eta_k$ which are $1$ on $P_k(a_k)$ and $0$ outside $P_k,$
and satisfy 
\begin{equation} \label{decay_cut_off}
\int_{\Sigma} |\nabla \eta_k|^2 dv_{g_k} \to 0.
\end{equation}
We denote by $\mathcal{C}=(-\infty,\infty) \times \Sph^1$ the infinite cylinder
with its canonical coordinates $(t, \theta) \in (-\infty, \infty) \times [0,2\pi)$.
Let $\phi:\C\rightarrow \Sph^2\subset\IR^3$ be given by
\begin{align*}
\phi(t, \theta)=\frac{1}{e^{2t}+1}(2e^t\cos(\theta),2e^t\sin(\theta),e^{2t}-1).
\end{align*}
Observe that this induces a map $\psi \colon \mathcal{M} \to \mathbb{RP}^2(\sqrt{3})$ if we divide by the $\IZ/2$-actions that we have on both sides.
More precisely, $\mathcal{M}=\mathcal{C}/\sim,$ where
$(t,\theta) \sim (-t,\theta+\pi)$ as above, and on $\Sph^2$ we simply take the antipodal map.
If we denote by $v:\mathbb{RP}^2(\sqrt{3})\rightarrow\Sph^4$ the Veronese map, the concatenation $v\circ\phi: \mathcal{M} \rightarrow\Sph^4$ is a conformal map \cite{girouard}.
We may regard $\mathcal{M}_k \subset \mathcal{M}$ using Fermi coordinates as introduced above.
\smallskip

By a theorem of Hersch \cite{hersch}, there exists a conformal diffeomorphism $\tau_{k}$ of $\Sph^4$, such that 
\begin{align*}
\int_{P_{k}}(\pi\circ\tau_{k}\circ v\circ\phi)\eta_{k}dv_{g_{k}}=0,
\end{align*}
where $\pi:\Sph^4\hookrightarrow\IR^5$ is the standard embedding.
Set $u_{k}^i=(\pi_i\circ\tau_{k}\circ v\circ\phi)\eta_{k}$.
By construction, we have
\begin{align*}
\sum_{i=1}^5\int_{\mathcal{M}_{k}}(u_{k}^i)^2dv_{g_{k}}\geq 1-\frac{D}{a_{k}},
\end{align*}
since $\area_{g_{k}}(P_{\alpha}^i(a_{k}))\geq 1-\tfrac{D}{a_{k}}$. 
Using conformal invariance and \eqref{decay_cut_off}, one easily finds that
\begin{align*}
\int_{\Sigma}\lvert \nabla u_{k}\lvert^2_{g_{k}}dv_{g_{k}}\leq 12\pi+o(1).
\end{align*}
For details we refer to \cite{girouard}.
Consequently, there is $i=i(k) \in \{1,\dots,5\},$ such that
\begin{align*}
 \lambda_1(\Sigma,g_k) \leq
\frac{\int_{M}\lvert \nabla u_{k}^i\lvert^2_{g_{k}}dv_{g_{k}}}{\int_{M} (u_{k}^i)^2dv_{g_{k}}}
 \leq12\pi+o(1).
\end{align*}
This finally implies
\begin{equation*}
\Lambda_1^K(\delta)\leq \limsup_{k \to \infty} \lambda_1(\Sigma,g_k) \leq 12\pi,
\end{equation*}
which establishes the claim.
\qedhere
\end{proof}

We are thus left with the case second case from  Lemma\,\ref{vol_cases}. In this case, we have the following lemma, which concludes the proof of Proposition\,\ref{degen_prop}.

\begin{lemma}
Suppose that the second alternative from \ref{vol_cases} holds, then
either 
\begin{itemize}
\item[(i)]
	$\Lambda_1^K(\delta) \leq \Lambda_1^K(\delta-1)$, or
\item [(ii)]$\Lambda_1^K(\delta) \leq \Lambda_1(\gamma),$
\end{itemize}
where $\gamma=\lfloor (\delta-1)/2 \rfloor.$
\end{lemma}

\begin{proof}
Again, we apply the machinery from \cite{petrides} to the orientation cover.
The essential point is to keep track of the geometry of the corresponding involutions.
Denote by $(\hat \Sigma, \hat h_k)$ the orientation covers of $(\Sigma,h_k),$
and by $\iota_k$ the corresponding deck transformations.

We can then identify the spectrum of the Laplacian for any metric $g$ in $[h_k]$ with the spectrum
of the Laplacian acting only on the even functions on $(\hat \Sigma, \hat g).$
We consider the associated harmonic maps $\Phi_k \colon (\Sigma, g_k) \to \Sph^l.$
By conformal invariance, we can also view these as harmonic maps from $(\Sigma, h_k)$ to $\Sph^l$.
In this situation, the metric can be recovered by
\begin{equation*}
g_k = \frac{|\nabla \Phi_k|_{h_k}^2}{\lambda_1(\Sigma,g_k)}h_k,
\end{equation*}
see \cite[Proof of Theorem 1]{petrides}.
By pulling back the $\Phi_k$'s to $\hat \Sigma,$
we obtain even harmonic maps $\hat \Phi_k \colon (\hat \Sigma, \hat h_k) \to \Sph^l,$
such that
\begin{equation*}
\hat g_k = \frac{|\nabla \hat \Phi_k |^2_{\hat h_k}}{\lambda_1(\Sigma, g_k)} \hat h_k.
\end{equation*}
With out loss of generality, we may assume that the volume (with respect to $g_k$) concentrates in $\Sigma_k^1(9a_k).$
Denote by $\hat \Sigma_k^1(9 a_k)$ its preimage under the covering projection.
Note that this preimage might be disconnected.
As in \cite[Sect.\ 4]{petrides}, there are a compact Riemann surface $\bar \Sigma$ 
and diffeomorphisms $\tau_k \colon \bar \Sigma \setminus \{p_1,\dots,p_r \} \to \hat \Sigma_k^1(9a_k).$
Moreover, the hyperbolic metrics $\bar h_k=\tau_k^* \hat h_k$ converge in $C^\infty_{loc}(\bar \Sigma \setminus \{p_1,\dots,p_r\})$ to a hyperbolic metric $\bar h.$

Observe, that  we can restrict and pullback the involutions $\iota_k$
to get involutions $\bar \iota_k$ of $\bar \Sigma \setminus \{p_1,\dots, p_r\}.$
Clearly, these involutions are isometric with respect to the hyperbolic metrics $\bar h_k.$

In a next step, we construct a fixed point free limit involution on $\bar \Sigma.$
For the compact subsets $\bar \Sigma_c:=\{x \in \bar \Sigma \ | \ \inj_x(\bar \Sigma,\bar h)\geq c \},$ we can
argue exactly as in the proof of Lemma\,\ref{limit_inv} to get limit involutions $\bar \sigma_n$
on $\bar \Sigma_{1/n}.$
Since any isometric involution must map $\bar \Sigma_c$ to itself, we may take subsequences, such that for $m\geq n,$ we have
$\left. \bar \sigma_m \right|_{\bar \Sigma_{1/n}}=\bar \sigma_n.$
Using a standard diagonal argument, we find a limit involution on  $\bar \Sigma \setminus \{p_1,\dots,p_r\}.$
Clearly, this involution extends to an involution $\bar \iota$  on all of $\bar \Sigma.$
Moreover, $\bar \iota$ is fixed point free:
Arguing again as in Lemma\,\ref{limit_inv}, we can not have fixed points different from the $p_i$'s.
If say $p_1$ is fixed under $\bar \iota,$ the involution is just rotation by $\pi$ in a disc centered at $p_1.$
By $C^0$-convergence away from $p_1,$ we see that the involutions $\hat \iota_k$ act just via rotation on the collars around the degenerating geodesic.
But this is impossible, since this implies that $\hat \iota_k$ is orientation preserving.

By \cite{zhu}, the pullbacks $\bar \Phi_k$ of the harmonic maps $\hat \Phi_k$ along the diffeomorphisms
$\tau_k$ are then harmonic maps that converge in
$C^\infty_{loc}(\bar \Sigma \setminus \{ p_1, \dots, p_r, x_1,\dots, x_s\})$ to a limit harmonic map $\bar \Phi$. 
Clearly, $\bar \Phi$ is invariant under $\bar \iota.$
Note, that no energy can be lost at the points $x_i$ or $p_i$.
By construction, no volume concentrates near the closed geodesics bounding $\hat \Sigma^1_k(9a_k) \subset \hat \Sigma$, which implies that no energy is lost
at the points $p_i$:
Observe next, that the points $x_i$ always come in pairs by the invariance of the harmonic maps.
Moreover, from the construction of the limit involution, it is clear, that two such points are bounded away from each other.
Therefore, energy concentration of the harmonic maps in a point $x_i$ implies that the volume with respect to the metric $g_k$ concentrates at a point in $\Sigma$.
But by \cite[Lemma 2.1 and 3.1]{kokarev} this implies
\begin{equation*}
\Lambda_1^K(\delta)=\lim_{k \to \infty}(\Sigma , g_k) \leq 8 \pi.
\end{equation*}
Finally, the energy identity from \cite{zhu} implies that there is also no energy lost in the necks.
Let $\bar h_0$ be the hyperbolic metric in the conformal class of the cusp compactification of $(\bar \Sigma \setminus\{p_1,\dots, p_r\}, [\bar h])$.
Since $\bar \Phi \colon (\bar \Sigma \setminus\{p_1,\dots, p_s\}, [\bar h_0]) \to \mathbb{S}^l$ has finite energy, $\bar \Phi$ extends to a harmonic map $(\bar \Sigma,[\bar h_0]) \to \Sph^l$ \cite[Theorem 3.6]{sacks_uhlenbeck}.
Moreover, this extension is certainly invariant under $\bar \iota.$ 
In conclusion, $\bar \Phi \colon (\bar \Sigma , \bar h_0) \to \mathbb{S}^l$ is a $\bar \iota$ invariant harmonic map with energy 
\begin{equation*}
\int_{\bar \Sigma} |\nabla \bar \Phi|^2 dv_{\bar h_0} = \lim_{k \to \infty} \int_{\hat \Sigma} |\nabla \hat \Phi_k| dv_{\hat h_k}.
\end{equation*}

We consider the metric
\begin{equation*}
\bar g = \frac{|\nabla \bar \Phi|_{\bar h_0}^2}{\Lambda_1^K(\delta)}\bar h_0
\end{equation*}
and observe that it is invariant under the involution $\bar \iota,$
so that it descends to a metric $g$ on $\bar \Sigma / \bar \iota.$
Since there is no energy lost along the sequence $\bar \Phi_k$ of harmonic maps, we have
\begin{equation*}
\area(\bar \Sigma / \bar \iota,  g)=1.
\end{equation*}
Using that the capacity of a point relative to any ball is $0$ \cite[Chapter 2.2.4]{mazya}, it is easy to construct $\bar \iota_k$-invariant cut-offs $\eta_{\eps,k}$ on $\bar \Sigma$
with the following two properties.
For $\eps$ small, there are neighborhoods $U_\eps \subset V_\eps$ of $\{p_1,\dots, p_r\}$ such that 
\begin{itemize}
\item
$\cap_{\eps>0} V_\eps = \{ p_1,\dots, p_r\},$
\item  $\eta_{\eps,k}=0$ in $U_\eps,$ 
\item $\eta_{\eps,k}=1$ outside $V_\eps,$
and
\item
$\int_{\bar \Sigma} |\nabla \eta_{\eps,k}|^2 dv_{\bar g}\leq \eps^2.$
\end{itemize}

We write $\bar g_k = \tau_k^*(\hat g_k)$.
Let $u$ be the lift of a first eigenfunction of $(\bar \Sigma/\bar \iota, g)$ to $\bar \Sigma$.
Using $\eta_{\eps,k} u$ as a test function on $\hat \Sigma_k(9a_k)$ for $k$ large enough, we find with the help of the dominated convergence theorem, that
\begin{equation*}
\begin{split}
\Lambda_1^K(\delta) 
& =\lim_{k \to \infty}\lambda_1(\Sigma,g_k) 
\\
&\leq \limsup_{\eps \to 0} \lim_{k \to \infty} \frac{\int_{\bar \Sigma} |\nabla (\eta_{\eps,k} u)|^2 dv_{\hat g_k}}{\int_{\bar \Sigma}|\eta_{\eps,k} u|^2 dv_{\bar g_k} - \left(\int_{\bar \Sigma} \eta_{\eps,k} u dv_{\bar g_k} \right)^2}
\\
& \leq \limsup_{\eps \to 0} \frac{\int_{\bar \Sigma} |\nabla u|^2 dv_{\bar g} + C\eps}{\int_{\bar \Sigma \setminus V_\eps}|u|^2 dv_{\bar g} - \left(\int_{\bar \Sigma } u dv_{\bar g} \right)^2}
\\
& \leq \frac{\int_{\hat \Sigma} |\nabla u|^2 dv_{\bar g}}{\int_{\bar \Sigma}|u|^2 dv_{\bar g}}
\\
 &\leq \lambda_1(\bar \Sigma / \bar \iota, g).
\end{split}
\end{equation*}
If $\bar \Sigma$ is disconnected, it has two connected components and the genus of each component is at most
$\lfloor (\delta -1)/2 \rfloor.$
Therefore, the quotient $\bar \Sigma / \bar \iota$ is an orientable surface of genus at most $\lfloor (\delta-1)/2\rfloor$ in this case
and thus 
$$
\Lambda_1^K(\delta) \leq \Lambda_1(\lfloor (\delta-1)/2\rfloor)
$$
thanks to \cite{ces}.
In case $\bar \Sigma$ is connected, the quotient is non-orientable of non-orientable genus at most $\delta-1$
and we have that
$$
\Lambda_1^K(\delta) \leq \Lambda_1^K(\delta-1)
$$ 
again thanks to \cite{ces}\footnote{More precisely, the non-orientable version of the result.}.
\end{proof}

Thanks to the weak inequality
$
\Lambda_1^K(\delta+1) \geq \Lambda_1^K(\delta-1)
$
and the fact that  $\Lambda_1^K(2)>12 \pi,$ we can always rule out the first scenario from \cref{vol_cases}.
Thanks to \cite{li_yau} and \cite{ckm,esgj, nadirashvili} our main result \cref{ex_non_orientable_main} follows from the Theorem below.

\begin{theorem} \label{existence}
Let $\delta \geq 3.$
If $\Lambda_1^K(\delta)>\max\{\Lambda_1^K(\delta-1), \Lambda_1(\lfloor (\delta-1)/2\rfloor)\},$
there is a metric smooth away from finitely many singularities on $\Sigma_\delta^K$ that achieves $\Lambda_1^K(\delta).$
\end{theorem}

\begin{proof}
By the assumptions, \cref{comp_non_or}, and \cref{degen_prop}, we can take hyperbolic metrics $h_k \to h$ in $C^\infty,$ such that
\begin{equation*}
\lim_{k \to \infty} \sup_{g \in [h_k]} \lambda_1(\Sigma , g) \area(\Sigma,g)
=\Lambda_1^K(\delta).
\end{equation*}
As above, we take unit volume metrics $g_k \in [h_k],$ such that 
$$\lambda_1(\Sigma,g_k)=\sup_{g \in [h_k]} \lambda_1(\Sigma,g)\area(\Sigma,g).$$
For the corresponding sequence of harmonic maps $\Phi_k \colon (\Sigma,h_k) \to \Sph^l$ no bubbling can occur
since this would imply $\Lambda_1^K(\delta)\leq 8\pi,$ by the same argument as above.
Therefore, we can take a subsequence such that $\Phi_k \to \Phi$ in $C^\infty,$
which implies that $g_k \to g=\frac{|\nabla \Phi|_h^2}{\Lambda_1^K(\delta)} h$ in $C^\infty.$
In particular,
\begin{equation*}
\lambda_1(\Sigma, g)\area(\Sigma,g)=\Lambda_1^K(\delta)
\end{equation*}
and $g$ is smooth away from the branch points of $\Phi.$
The number of branch points is finite and the branch points correspond to conical singularities of $g$ \cite{salamon}.
\end{proof}

\appendix 

\section{Topology of surfaces} \label{notation}

For convenience of the reader and the authors, we review here the notion of non-orientable genus.

Recall the classification of closed surfaces.
The classes of closed orientable and non-orientable surfaces are both uniquely described up to diffeomorphism by the Euler characteristic.
More precisely, any closed orientable surface is diffeomorphic to a surface of the form
\begin{align*}
\Sigma_\gamma = \Sph^2 \# \underbrace{T^2 \#  \dots \# T^2}_{\gamma-\text{times}},
\end{align*}
and any closed non-orientable surface is diffeomorphic to a surface of the form
\begin{align*}
\Sigma_\delta^K = \Sph^2 \# \underbrace {\IR P^2 \# \dots \# \IR P^2}_{\delta-\text{times}}.
\end{align*}
These two families provide -- up to diffeomorphism --  a complete list of all orientable respectively non-orientable closed surfaces.
We call $\gamma$ the \emph{genus} of $\Sigma_\gamma$ and $\delta$ the \emph{non-orientable genus} of $\Sigma_\delta^K.$
Note that with this convention, the real projective plane has non-orientable genus $1$.
We have  $\chi(\Sigma_{\gamma})=2-2{\gamma}$ and $\chi(\Sigma_\delta^K)= 2-\delta,$ so that the orientation
cover of $\Sigma_\delta^K$ is given by $\Sigma_{\delta-1}.$
Some authors prefer to refer to the genus of the orientation cover as the non-orientable genus.
As explained above these two definitions differ.
Moreover, recall that we have the relation
\begin{align*}
	\Sph^2 \# \underbrace {\IR P^2 \# \dots \# \IR P^2}_{\delta-\text{times}} 
	\cong
	\Sph^2 \# \underbrace{T^2 \# \dots \# T^2}_{k-\text{times}} \# \underbrace {\IR P^2 \# \dots \# \IR P^2}_{(\delta-2k)-\text{times}},
\end{align*} 
if $2k<\delta.$

\end{document}